\documentclass[psamsfonts,fceqn,leqno]{amsart}
\usepackage{mathrsfs,latexsym,amsfonts,amssymb,curves,epic}
\usepackage{color}

\newtheorem{theorem}{Theorem}[section]
\newtheorem{corollary}[theorem]{Corollary}
\newtheorem{proposition}[theorem]{Proposition}
\newtheorem{lemma}[theorem]{Lemma}
\newtheorem{question}[theorem]{Question}

\theoremstyle{definition}
\newtheorem{definition}[theorem]{Definition}

\begin{document}
\noindent \vspace{0.5in}

\title[Locally $\sigma$-compact rectifiable spaces]
{Locally $\sigma$-compact rectifiable spaces}

\author{Fucai Lin}
\address{(Fucai Lin): School of mathematics and statistics,
Minnan Normal University, Zhangzhou 363000, P. R. China}
\email{linfucai2008@aliyun.com}

\author{Jing Zhang}
\address{(Jing Zhang) School of mathematics and statistics,
Minnan Normal University, Zhangzhou 363000, P. R. China}
\email{zhangjing86@126.com}

\author{Kexiu Zhang}
\address{(Kexiu Zhang) School of mathematics and statistics,
Minnan Normal University, Zhangzhou 363000, P. R. China}
\email{kxzhang2006@163.com}

\thanks{Supported by the NSFC (Nos. 11201414, 11471153), the Natural Science Foundation of Fujian Province (No. 2012J05013) of China
and Training Programme Foundation for Excellent Youth Researching Talents
of Fujian's Universities (JA13190).}


\begin{abstract}
A topological space $G$ is said to be a {\it rectifiable space}
provided that there are a surjective homeomorphism $\varphi :G\times
G\rightarrow G\times G$ and an element $e\in G$ such that
$\pi_{1}\circ \varphi =\pi_{1}$ and for every $x\in G$,
$\varphi (x, x)=(x, e)$, where $\pi_{1}: G\times G\rightarrow G$ is
the projection to the first coordinate. In this paper, we first prove that each locally compact rectifiable space is paracompact, which gives an affirmative answer to Arhangel'skii and Choban's question (Arhangel'skii and Choban \cite{A2009}). Then we prove that every locally $\sigma$-compact rectifiable space with a $bc$-base is
locally compact or zero-dimensional, which improves Arhangel'skii and van Mill's result (Arhangel'skii and van Mill \cite{A2014}). Finally, we prove that each $k_{\omega}$-rectifiable space is rectifiable complete.
\end{abstract}

\maketitle

\section{Introduction}
Recall that a {\it paratopological group} $G$ is a group $G$ with a
 topology such that the product mapping of $G \times G$ into
$G$ is jointly continuous. The space $G$ is called a {\it topological group} if it is a paratopological group and the inverse mapping of $G$ onto itself
associating $x^{-1}$ with arbitrary $x\in G$ is continuous. In addition,
the space $G$ is said to be a {\it rectifiable space}
provided that there are a surjective homeomorphism $\varphi :G\times
G\rightarrow G\times G$ and an element $e\in G$ such that
$\pi_{1}\circ \varphi =\pi_{1}$ and for every $x\in G$,
$\varphi (x, x)=(x, e)$, where $\pi_{1}: G\times G\rightarrow G$ is
the projection to the first coordinate. If $G$ is a rectifiable
space, then $\varphi$ is called a {\it rectification} on $G$. It is
well known that rectifiable spaces and paratopological groups are
two good generalizations of topological groups. In fact, for a
topological group with the neutral element $e$, it is easy to
see that the mapping $\varphi (x, y)=(x, x^{-1}y)$ is a rectification on
$G$. However, the 7-dimensional sphere $S_{7}$ is
rectifiable but not a topological group \cite[$\S$ 3]{V1990}. The Sorgenfrey line $G$ is a paratopological group with no rectification on $G$.
Further, it is easy to see that
each rectifiable space is homogeneous. Recently, the study of rectifiable spaces has become an interesting topic in topological algebra, see \cite{A20092, A2009, G1996, HK, LF4, LF3, LF2, LF1, LLL}. In particular,
the following theorem plays an important role in the study of rectifiable spaces.

\begin{theorem}\cite{C1992, G1996, V1989}\label{t4}
A topological space $G$ is rectifiable if and only if there exist an element $e\in G$ and two
continuous mappings $p: G^{2}\rightarrow G$, $q: G^{2}\rightarrow G$
such that for any $x, y\in G$ the next
identities hold:
\smallskip
$$p(x, q(x, y))=q(x, p(x, y))=y\ \mbox{and}\ q(x, x)=e.$$
\end{theorem}

Given a rectification $\varphi$ of the rectifiable space $G$, we can obtain the mappings $p$ and $q$ in Theorem~\ref{t4} as follows. Let $p=\pi_{2}\circ\varphi^{-1}$ and
$q=\pi_{2}\circ\varphi$. Then the mappings $p$ and $q$ satisfy the identities in
Theorem~\ref{t4}, and both are open mappings.

Let $G$ be a rectifiable space, and let $p$ be the multiplication on
$G$. Then we sometimes write $x\cdot y$ instead of $p(x, y)$ and
$A\cdot B$ instead of $p(A, B)$ for any $A, B\subset G$. Obviously,
$q(x, y)$ is an element in $G$ such that $x\cdot q(x, y)=y$. Moreover, since $x\cdot
e=x\cdot q(x, x)=x$ and $x\cdot q(x, e)=e$, it follows that $e$ is a right neutral
element for $G$ and $q(x, e)$ is a right inverse for $x$. Therefore, a
rectifiable space $G$ is a topological algebraic system with
operations $p$ and $q$, a 0-ary operation $e$, and identities as above. Further, the multiplication operation $p$ of this algebraic system need not satisfy the
associative law. Note that
every topological loop is rectifiable.

In this paper, we mainly discuss the topological properties of locally $\sigma$-compact rectifiable spaces. The paper is organized as follows:

In Section 2, we mainly prove that each locally compact rectifiable space is paracompact, which gives an affirmative answer to a question of A.V. Arhangel'skii and M.M. Choban.

In Section 3, we mainly prove that every locally $\sigma$-compact rectifiable space with a $bc$-base is
locally compact or zero-dimensional, which improves a result of A.V. Arhangel'skii and J. van Mill.

In Section 4, we mainly prove that each $k_{\omega}$-rectifiable space is rectifiable complete.

All the spaces considered in this paper are supposed to be Hausdorff unless stated otherwise.
The notation $\omega$ denotes the first countable infinite order. The letter $e$
denotes the right neutral element
of a rectifiable space. For undefined notation and terminologies, the reader may refer to
\cite{A2008} and \cite{E1989}.

\section{locally compact rectifiable spaces}
In this section, we shall discuss the locally $\sigma$-compact rectifiable spaces, and  show that each locally $\sigma$-compact rectifiable space is paracompact. 

In \cite{A2009}, A.V. Arhangel'skii and M.M. Choban posed the following two questions:

\begin{question}\cite[Problem 5.9]{A2009}\label{q1}
Is every rectifiable $p$-space paracompact? What if the space is locally compact?
\end{question}

\begin{question}\cite[Problem 5.10]{A2009}\label{q2}
Is every rectifiable $p$-space a $D$-space?
\end{question}

It is well known that a paracompact $p$-space is a $D$-space. Therefore, if we can prove each rectifiable $p$-space is paracompact, then the answer to Question~\ref{q2} is also positive. However, we will prove that each locally compact rectifiable space is paracompact, which gives
an affirmative answer to the second part of Question~\ref{q1}, see Corollary~\ref{c0}. By this result, we give a partial answer to Question~\ref{q2}, see Corollary~\ref{c3}. First of all, we give some concepts and technical lemmas.

A space $X$ is called {\it locally compact} (resp. {\it locally $\sigma$-compact}) if, for each $x\in X$, there exists a neighborhood $U$ at $x$ in $X$ such that the closure of $U$ in $X$ is  compact (resp. $\sigma$-compact).

\begin{definition}\cite{LLL}
Let $H$ be a subspace of a rectifiable space $G$. Then $H$ is called {\it a rectifiable subspace of $G$} if we have $p(H, H)\subset H$ and $q(H, H)\subset H$.
\end{definition}

\begin{lemma}\label{p0}
If $A$ is subset of a rectifiable space $G$, then $H=\bigcup_{n\in\mathbb{N}}(A_{n}\cup B_{n})$ is the smallest rectifiable subspace containing the set $A$, where $A_{1}=A, B_{1}=q(A, e)\cup q(A, A)$, $A_{n+1}=p(A_{n}\cup B_{n}, A_{n}\cup B_{n})$ and $B_{n+1}=q(A_{n}\cup B_{n}, A_{n}\cup B_{n})$ for each $n\in\mathbb{N}.$ Moreover, $H$ is open if $A$ is an open subset in $G$.
\end{lemma}

\begin{proof}
Since $q(A, A)\subset B_{1}$, we have $e\in B_{1}$. Therefore, it is easy to see that $A_{n}\cup B_{n}\subset A_{n+1}\cup B_{n+1}$ for each $n\in\mathbb{N}$. Next we shall prove that $H$ is a rectifiable subspace of $G$. Indeed, take arbitrary points $x, y\in H$. Then there exists an $n\in \mathbb{N}$ such that $x, y\in A_{n}\cup B_{n}$, and hence $p(x, y)\in A_{n+1}\cup B_{n+1}$ and $q(x, y)\in A_{n+1}\cup B_{n+1}$. Therefore, $H$ is a rectifiable subspace of $G$.

Obviously, it follows from the construction of $H$ that $H$ is open if $A$ is an open subset in $G$.
\end{proof}

\begin{lemma}\cite{LLL}\label{l1}
Let $G$ be a rectifiable space. If $H$ is an open rectifiable subspace of $G$, then $H$ is closed in $G$.
\end{lemma}

\begin{lemma}\cite{LLL}\label{l2}
Let $G$ be a rectifiable space. If $H$ is a rectifiable subspace of $G$, then the closure of $H$ in $G$  is a rectifiable subspace.
\end{lemma}

\begin{proposition}
Let $G$ be a rectifiable space. If $H$ is a locally compact rectifiable subspace of $G$, then $H$ is closed in $G$.
\end{proposition}

\begin{proof}
Let $H_{1}$ be the closure of $H$ in $G$. Then it follows from Lemma~\ref{l2} that $H_{1}$ is a rectifiable subspace of $G$. Moreover, $H$ is a locally compact subspace of $H_{1}$, which implies that it is open in $H_{1}$. By Lemma~\ref{l1}, $H$ is closed in $H_{1}$, hence it is closed in $G$.
\end{proof}

Recall that a topological space is {\it $\kappa$-Lindel\"{o}f}, where $\kappa$ is any cardinal, if every open cover has a subcover of cardinality strictly less than $\kappa$.
An $\aleph_{1}$-Lindel\"{o}f space is called {\it Lindel\"{o}f}.
The {\it Lindel\"{o}f degree}, or {\it Lindel\"{o}f number $l(X)$}, is the smallest cardinal $\kappa$ such that every open cover of the space $X$ has a subcover of size at most $\kappa$.
Recall that a space is {\it strongly paracompact}
if every open cover has a star-finite open refinement. Obviously, each strongly paracompact space is paracompact.

The following lemma is well known.

\begin{lemma}\label{p1}
Let $X$ be a compact space and $Y$ be a $\kappa$-Lindel\"{o}f space. Then $X\times Y$ is a $\kappa$-Lindel\"{o}f space.
\end{lemma}

\begin{theorem}\label{t1}
If a rectifiable space $G$ contains an open $\sigma$-compact rectifiable subspace $H$, then $G$ admits a disjoint open cover refining the open cover $\{g\cdot H: g\in G\}$. Therefore, $G$ is a topological sum of $\sigma$-compact subspaces and hence $X$ is strongly paracompact.
\end{theorem}

\begin{proof}
Let $l(G)=\gamma$. By induction of cardinals $\alpha\leq \gamma$ we will prove that each rectifiable subspace $X\subset G$ with $H\subset X$ and $l(X)\leq\alpha$ admits a disjoint open cover refining the open cover $\{x\cdot H: x\in X\}$.

Suppose that $X$ is a Lindel\"{o}f rectifiable subspace of $G$ such that $H\subset X$. Since $X$ is Lindel\"{o}f, there exists a sequence $\{x_{n}\}_{n\in\omega}\subset X$ such that $\{x_{n}\cdot H: n\in \omega\}$ is a countable subcover of the open cover $\{x\cdot H: x\in X\}$. For each $n\in\omega$, put
$$U_{n}=(x_{n}\cdot H)\setminus\bigcup_{k<n}x_{k}\cdot H. $$It follows from Lemma~\ref{l1} that $U_{n}$ is a clopen subset of $G$. Therefore, $\mathscr{U}=\{U_{n}: n\in \omega\}$ is a disjoint open cover of $X$ refining the cover $\{x\cdot H: x\in X\}$.

Now suppose that, for some uncountable cardinal $\kappa\leq\gamma$, we have proved that any rectifiable subspace $X\subset G$ with $H\subset X$ and $l(X)<\kappa$ admits a disjoint open cover refining the cover $\{x\cdot H: x\in X\}$. Next we will show that any rectifiable subspace $X$ with $H\subset X$ and $l(X)=\kappa$ admits a disjoint open cover refining the cover $\{x\cdot H: x\in X\}$.

Fix a rectifiable subspace $X\subset G$ with $H\subset X$ and $l(X)=\kappa$. Since $l(X)=\kappa$, we can take a subset $B\subset X$ such that the cardinality $|B|\leq\kappa$ and $X=\bigcup_{x\in B} x\cdot H$. Enumerate $B$ as $\{x_{\alpha}: \alpha<\kappa\}$. For each $\alpha<\kappa$, let $H_{\alpha}$ be the smallest rectifiable subspace of $G$ containing the set $B_{\alpha}=H\cup \{x_{\beta}\}_{\beta<\alpha}$. Since $H_{\alpha}\cdot H\subset H_{\alpha}$ and $H$ is open, the rectifiable subspace $H_{\alpha}$ is open in $G$. It follows from Lemma~\ref{l1} that $H_{\alpha}$ is closed in $G$. Since $H$ is $\sigma$-compact, it is easy to see that $l(H_{\alpha})\leq\alpha<\kappa$ by Lemma~\ref{p1} and the construction of $H_{\alpha}$ in Lemma~\ref{p0}. Hence by the inductive assumption, the rectifiable subspace $H_{\alpha}$ admits a disjoint open cover $\mathscr{U}_{\alpha}$ refining the open cover $\{x\cdot H: x\in H_{\alpha}\}$. Moreover, it easily checked that the union $H^{<\alpha}=\bigcup_{\beta<\alpha}H_{\beta}$ is an open rectifiable subspace of $G$. Then the rectifiable subspace $H^{<\alpha}$ is closed in $G$ by Lemma~\ref{l1}. Now, for each $\alpha<\kappa$, put $$\mathscr{V}_{\alpha}=\{U\setminus H^{<\alpha}: U\in \mathscr{U}_{\alpha}\}.$$ Obviously, each $\mathscr{V}_{\alpha}$ is a disjoint open cover of the space $H_{\alpha}\setminus H^{<\alpha}$ refining the cover $\mathscr{U}_{\alpha}$. Put $$\mathscr{V}=\bigcup_{\alpha<\kappa}\mathscr{V}_{\alpha}.$$ Then $\mathscr{V}$ is a disjoint open cover of $X$ and refines the open cover $\{x\cdot H: x\in X\}$ of $X$.

By induction, we can see that $G$ admits a disjoint open cover refining the open cover $\{g\cdot H: g\in G\}$. Since $H$ is open $\sigma$-compact, $G$ is a topological sum of $\sigma$-compact subspaces, then $X$ is strongly paracompact.
\end{proof}

\begin{corollary}
Let $G$ be a locally $\sigma$-compact rectifiable space. Then it is a paracompact space.
\end{corollary}

\begin{proof}
Let $U$ be an open neighborhood of $e$ in $G$ such that the closure of $U$ is $\sigma$-compact. Let $H$ be the smallest rectifiable subspace of $G$ containing the set $U$. It is easy to see that $H$ is open in $G$, hence it is closed in $G$ by Lemma~\ref{l1}. Moreover, it is easy to obtain that $H$ is $\sigma$-compact by Lemma~\ref{p0}. Then, by Theorem~\ref{t1}, $G$ is paracompact.
\end{proof}

\begin{corollary}\label{c0}
Let $G$ be a locally compact rectifiable space. Then it is a paracompact space.
\end{corollary}

\begin{corollary}\label{c1}
Each connected locally compact rectifiable space is $\sigma$-compact.
\end{corollary}

\begin{corollary}\label{c2}
Each locally compact rectifiable space is normal.
\end{corollary}

\begin{corollary}\label{c3}
Each locally compact rectifiable space is a $D$-space.
\end{corollary}

\section{Zero-dimensional rectifiable spaces}
In this section, we shall discuss rectifiable spaces with a $bc$-base. We mainly prove that for each rectifiable space with a $bc$-base $G$, either $G$ is locally compact or every compact subspace of $G$ is zero-dimensional.

We will call a base $\mathscr{B}$ of a space $X$ a {\it $bc$-base} if the boundary $B(U)=\overline{U}\setminus U$ of every
member $U$ of $\mathscr{B}$ is compact.

In \cite{A2014}, A.V. Arhangel'skii and  J. van Mill proved the following theorem:

\begin{theorem}\cite{A2014}
Every $\sigma$-compact non-locally compact topological group with a $bc$-base is
zero-dimensional.
\end{theorem}

Next, we will improve this theorem by replacing ``topological group'' with ``rectifiable space'', see Corollary~\ref{c4}.

We say that a space $X$ is {\it separated by a compact subset $F$ of $X$ between points $x$ and $y$} \cite{A2014}
of $X$ if there are disjoint open subsets $U$ and $V$ such that $x\in U, y\in V$ , and $U\cup V=X\setminus F$.
A space $X$ is {\it separated by compacta} \cite{A2014} if for any two distinct points $x, y\in X$, the space $X$ is
separated between $x$ and $y$ by some compact subspace of $X$.

A topological space $X$ is {\it totally disconnected} if all connected components in $X$ are the one-point sets.

The following proposition is well known.

\begin{proposition}\label{p2}
A totally disconnected locally compact space is zero-dimensional.
\end{proposition}

\begin{theorem}\label{t000}
Suppose that $G$ is a rectifiable space such that any
two distinct points of $G$ can be separated by a compactum. Then either $G$ is locally compact or every $\sigma$-compact subspace
of $G$ is zero-dimensional.
\end{theorem}

\begin{proof}
Suppose that $G$ is not locally compact. By the Countable Closed Sum Theorem of zero-dimensional spaces, it suffices to show that each compact subset of $G$ is zero-dimensional.

Take any compact subset $F$ of $G$. We will show that $F$ is zero-dimensional.
\smallskip

{\bf Claim 1:} For any compact subset $K$ of $G$ and any open neighborhood $U$ of $e$, there exists a point $c\in U$ such that $(F\cdot c)\cap K=\emptyset$.
\smallskip

Assume on the contrary that we have $(F\cdot c)\cap K\neq\emptyset$ for each $c\in U$. Hence there exist points $f\in G$ and $k\in K$ such that $fc=k$, and then $$c=q(f, p(f, c))=q(f, k)\in q(F, K).$$
Thus we have $U\subset q(F, K)$. Since $q$ is continuous, $q(F, K)$ is compact, hence $G$ is locally compact, which is a contradiction with our assumption.
\smallskip

{\bf Claim 2:} For any compact subset $K$ of $G$, any $x, y\in F$ and any open neighborhoods $U$ and $V$ of $x$ and $y$ in $G$, respectively, there exists a continuous mapping $h: F\rightarrow G$ such that $h(x)\in U$, $h(y)\in V$ and $K\cap h(F)=\emptyset$.
\smallskip

Indeed, since $U$ and $V$ are open neighborhoods of $x$ and $y$ in $G$ respectively, there exists an open neighborhood $W$ of $e$ in $G$ such that $$x\cdot W\subset U\ \mbox{and}\ y\cdot W\subset V.$$ By the Claim 1, there exists a $w\in W$ such that $(F\cdot w)\cap K=\emptyset$. Let $h: F\rightarrow G$ with $f\mapsto f\cdot w$ for each $f\in F$. Obviously, $h$ is continuous. Moreover, we have $$h(x)=x\cdot w\in x\cdot W\subset U\ \mbox{and}\ h(y)=y\cdot w\in y\cdot W\subset V.$$

From the Claim 2 and \cite[Proposition 3.4]{A2014}, $F$ is totally disconnected. Since $F$ is compact, it follows from Proposition~\ref{p2} that $F$ is zero-dimensional.
\end{proof}

By Theorem~\ref{t000}, the following two corollaries are obvious.

\begin{corollary}\label{c5}
Suppose that $G$ is a rectifiable space with a $bc$-base.
Then $G$ is locally compact or every compact subspace of $G$ is zero-dimensional.
\end{corollary}

\begin{corollary}
If $G$ is any non-locally compact rectifiable space with a $bc$-base, then
$\mbox{ind}(G) \leq 1$.
\end{corollary}

\begin{corollary}\label{c4}
Every locally $\sigma$-compact rectifiable space with a $bc$-base is
locally compact or zero-dimensional.
\end{corollary}

\begin{proof}
 By Lemma~\ref{p0} and Lemma~\ref{l1}, it is easy to see that $G$ has a $\sigma$-compact clopen rectifiable subspace. Obviously, it suffices to prove that $H$ is zero-dimensional. Therefore, without loss of generality, we may assume that $G$ is $\sigma$-compact. Finally, we only need to apply Corollary~\ref{c5}.
\end{proof}

Moreover, We do not know whether Corollary~\ref{c4} can be extended in the class of paratopological groups. Therefore, we have the following question.

\begin{question}
Is every locally $\sigma$-compact non-locally compact paratopological group with a $bc$-base
zero-dimensional?
\end{question}

It is well known that a totally disconnected locally compact topological group has an open compact subgroup. Therefore, we have the following question.

\begin{question}\label{q3}
Does each totally disconnected locally compact rectifiable space have an open compact rectifiable subspace?
\end{question}

Approximately 10 years ago, L.G. Zambakhidze asked whether every non-zero-dimensional topological group with a $bc$-base is locally compact. Recently, some progress has been made on this problem in \cite{A2014}. It is natural to ask the following question.

\begin{question}
Is every non-zero-dimensional paratopological group (or rectifiable space) with a $bc$-base locally compact?
\end{question}

It is well known that a locally compact paratopological group is a topological group, hence we have the following question.

\begin{question}
Is every non-zero-dimensional paratopological group with a $bc$-base topological group?
\end{question}

\section{$k_{\omega}$-rectifiable spaces}
In \cite{LF4}, Lin defined the concept of a rectifiable complete space and showed that each locally compact rectifiable space is  rectifiable complete, and hence each compact rectifiable space is rectifiable complete. In this section, we will prove that each $k_{\omega}$-rectifiable space $G$ is rectifiable complete. First of all, we recall some concepts and show some technical lemmas.

\begin{definition}
(1) A net $\{x_{\alpha}\}_{\alpha\in\Gamma}$ in a rectifiable space $G$ is a {\it left Cauchy net} if for every neighborhood $U$ of $e$ in $G$ there exists an $\alpha_{0}\in\Gamma$ such that $q(x_{\alpha}, x_{\beta})\in U$ for every $\alpha, \beta\geq\alpha_{0}.$

(2) A filter $\xi$ of subsets of a rectifiable space $G$ is said to be a {\it left Cauchy filter} if for every neighborhood of $e$ in $G$ there exist an $F\in\xi$ and a point $a\in G$ such that $q(a, F)\subset U$.

(3) A rectifiable space $G$ is called {\it rectifiable complete} if every left Cauchy filter in $G$ converges. Obviously, a rectifiable space $G$ is rectifiable complete if and only if every left Cauchy net in $G$ converges.

(4) Let $G$ be a rectifiable space. If there exists a rectifiable complete space $H$ such that $G$ is dense in $H$ then we say that $H$ is the rectifiable completion of $G$ and denotes it by $\widehat{G}$.
\end{definition}

\begin{definition}Let $G, H$ be two rectifiable spaces and $f: G\rightarrow H$ be a mapping from $G$ to $H$. The mapping is called {\it a homomorphism} if for arbitrary $x, y\in G$ we have $f(p_{G}(x, y))=p_{H}(f(x), f(y)).$ Moreover, if $f$ is an one-to-one homomorphism mapping from $G$ onto $H$, then $f$ is called an {\it isomorphism}.
\end{definition}

In \cite{LF4}, Lin proved that if a rectifiable space $G$ has two rectifiable complete $\widehat{G}$ and $\widehat{H}$ then $\widehat{G}$ and $\widehat{H}$ are topologically isomorphic.

We shall first prove the following lemma which plays an important role in our proof of our main theorem in this section.

\begin{lemma}\label{compact-separated}
Let $F$ and $K$ be two non-empty compact subsets of a rectifiable space $G$, and let $P$ be a closed subset of $G$. If $e\in K\subset F\cdot K$ and $F\cdot K\cap P=\emptyset$, then there exists an open neighborhood $U$ of $e$ such that $F\cdot ((K\cdot U)\cdot U)\cap P=\emptyset$.
\end{lemma}

\begin{proof}
{\bf Claim}: For each $y\in K$, there exists an open neighborhood $W_{y}$ of $e$ such that $F\cdot (y\cdot W_{y})\cap P=\emptyset$.
\smallskip

Indeed, fix a point $y\in K$. For each $x\in F$, since $p(x, y)\not\in P$, there exists an open neighborhood $W_{y}^{x}$ of $e$ such that $$[(x\cdot W_{y}^{x})\cdot (y\cdot W_{y}^{x})]\cap P=\emptyset.$$ Obviously, each $(x\cdot W_{y}^{x})\cdot (y\cdot W_{y}^{x})$ is an open neighborhood of $x\cdot y$. Then $$\{x\cdot W_{y}^{x}): x\in F\}\ \mbox{and}\ \{(x\cdot W_{y}^{x})\cdot (y\cdot W_{y}^{x}): x\in F\}$$ are open covers of $F$ and $F\cdot y$, respectively. It follows from the compactness of $F$ and $F\cdot y$, there exists a finite subset $C_{y}$ of $F$ such that $$F\subset \bigcup_{x\in C_{y}}x\cdot W_{y}^{x}\  \mbox{and}\ F\cdot y\subset \bigcup_{x\in C_{y}}((x\cdot W_{y}^{x})\cdot (y\cdot W_{y}^{x})).$$ Put $W_{y}=\bigcap_{x\in C_{y}} W_{y}^{x}$. Then $F\cdot (y\cdot W_{y})\cap P=\emptyset.$
Obviously, it suffices to prove that $$z\cdot (y\cdot W_{y})\cap P=\emptyset$$ for each $z\in F$. Since $F\subset \bigcup_{x\in C_{y}}x\cdot W_{y}^{x}$, there exists a point $x\in C_{y}$ such that $z\in x\cdot W_{y}^{x}$. Then
$$z\cdot (y\cdot W_{y})\subset (x\cdot W_{y}^{x})\cdot (y\cdot W_{y})\subset (x\cdot W_{y}^{x})\cdot (y\cdot W_{y}^{x})\subset G\setminus P. $$ Therefore, the Claim holds.

For each $y\in K$, since $p(y, e)=y$, it follows from the Claim that there exist open neighborhoods $U_{y}$ and $V_{y}$ of $e$ satisfying the following conditions:
\smallskip

(1) $U_{y}\subset V_{y}$;
\smallskip

(2) $V_{y}\subset W_{y}$;
\smallskip

(3) $F\cdot (y\cdot W_{y})\cap P=\emptyset$;
\smallskip

(4) $(y\cdot V_{y})\cdot V_{y}\subset y\cdot W_{y}$;
\smallskip

(5) $(y\cdot U_{y})\cdot U_{y}\subset y\cdot V_{y}$.
\smallskip

Then the family $\{F\cdot (y\cdot U_{y}): y\in K\}$ and $\{y\cdot U_{y}: y\in K\}$ are open covers of $F\cdot K$ and $K$ respectively.  It follows from the compactness of $F\cdot K$ and $K$ that there exists a finite subset $D$ of $K$ such that $$F\cdot K\subset \bigcup_{y\in D}F\cdot (y\cdot U_{y})$$ and $$K\subset  \bigcup_{y\in D}y\cdot U_{y}.$$ Put $U=\bigcap_{y\in D} U_{y}$. Then $F\cdot ((K\cdot U)\cdot U)\cap P=\emptyset.$

Indeed, it suffices to prove that $$F\cdot ((h\cdot U)\cdot U)\cap P=\emptyset$$ for each $h\in K$. Since $K\subset \bigcup_{y\in D}x\cdot U_{y}$, there exists a point $y\in D$ such that $h\in y\cdot U_{y}$. Then
\begin{eqnarray}
F\cdot ((h\cdot U)\cdot U)&\subset&F\cdot (((y\cdot U_{y})\cdot U)\cdot U) \nonumber\\
&\subset&F\cdot (((y\cdot U_{y})\cdot U_{y})\cdot U) \nonumber\\
&\subset&F\cdot ((y\cdot V_{y})\cdot U) \nonumber\\
&\subset&F\cdot ((y\cdot V_{y})\cdot V_{y}) \nonumber\\
&\subset&F\cdot (y\cdot W_{y}) \nonumber\\
&\subset&G\setminus P. \nonumber
\end{eqnarray}
\end{proof}

The following lemma was proved in \cite{PL}.

\begin{lemma}\label{dense}\cite{PL}
Let $G$ be a rectifiable space. If $A\subset G$ and $V$ is an open
neighborhood of the right neutral element $e$ of $G$, then $\overline{A}\subset A\cdot V$.
\end{lemma}

We say that the topology  of a space $X$ is {\it determined by a family $\mathscr{C}$} of its subsets provided that a set $F\subset X$ is closed in $X$ iff $F\cap C$ is closed in $C$, for each $C\in\mathscr{C}$. If the topology of a space $X$ is determined by an increasing sequence $\{K_{n}: n\in\omega\}$ of compact subsets in $X$, then $X$ is called a $k_{\omega}$-space.

\begin{theorem}
Let $G$ be a rectifiable space. If $G$ is a $k_{\omega}$-space, then $G$ is a rectifiable complete space.
\end{theorem}

\begin{proof}
Assume on the contrary that the rectifiable space $G$ fails to rectifiable complete. Then $G$ contains a left Cauchy filter $\zeta$ of closed subsets with empty intersection. Since $G$ is a $k_{\omega}$-space, there exists a sequence $\{K_{n}: n\in\omega\}$ of compact subsets in $G$ such that the topology of $G$ is determined by the family $\{K_{n}: n\in\omega\}$.
\smallskip

Let $$L_{0}=p(K_{0}, K_{0})\cup q(K_{0}, K_{0})$$ and $$L_{1}=p(K_{1}, K_{1})\cup q(K_{1}, K_{1})\cup p(L_{0}, L_{0})\cup q(L_{0}, L_{0})\cup p(K_{1}, L_{0})\cup q(K_{1}, L_{0}).$$ By induction, we can construct a sequence $\{L_{n}: n\in\omega\}$ of sets in $G$ such that, for each $n\geq 1$, we have $$L_{n}=p(K_{n}, K_{n})\cup q(K_{n}, K_{n})\cup p(L_{n-1}, L_{n-1})\cup q(L_{n-1}, L_{n-1})\cup p(K_{n}, L_{n-1})\cup q(K_{n}, L_{n-1}).$$

Then we can construct a sequence $\{K_{n}^{\ast}: n\in\omega\}$ of sets in $G$ such that $K_{0}^{\ast}=L_{0}$ and, for each $n\geq 1$, we have
\smallskip
$$K_{n}^{\ast}=p(L_{n}, K_{n-1}^{\ast}).$$Moreover, it is easy to see that $p(K_{n}^{\ast}, K_{n}^{\ast})\subset K_{n+1}^{\ast}$ for each $n\in\omega$.
\smallskip

Since each $K_{n}\subset K_{n}^{\ast}$, the compact sets $K_{n}^{\ast}$ generates the original topology of $G$, and hence there exists an element $C_{n}\in\zeta$ such that $K_{n}^{\ast}\cap C_{n}=\emptyset$. Let us construct a sequence $\{V_{i}: i\in\omega\}$ of sets in $G$ satisfying the following conditions for each $i\in\omega$:
\smallskip

(1) $e\in V_{i}\subset K_{i}^{\ast}$ and $V_{i}$ is open in $K_{i}^{\ast}$;
\smallskip

(2) $V_{i}\subset V_{i+1}$;
\smallskip

(3) $(K_{j}^{\ast}\cdot \overline{V_{i}})\cap C_{j+1}=\emptyset$, for each $j\leq i.$
\smallskip

By Lemma~\ref{compact-separated}, there exists an open neighborhood $U_{0}$ of $e$ in $G$ such that $$(K_{0}^{\ast}\cdot \overline{U_{0}})\cap C_{1}=\emptyset.$$ Put $V_{0}=U_{0}\cap K_{0}^{\ast}$. Assume that we have defined the sets $V_{0}, \cdots, V_{n}$ satisfying (1)-(3). It follows from (3) that $$(K_{i}^{\ast}\cdot \overline{V_{n}})\cap C_{i+1}=\emptyset$$ for each $i=0, 1, \cdots, n.$ Moreover, it easy to see that each $(K_{n+1}^{\ast}\cdot \overline{V_{n}})\subset K_{n+2}^{\ast}$, so that the choice of the set $C_{n+2}$ implies that $$(K_{n+1}^{\ast}\cdot \overline{V_{n}})\cap C_{n+2}=\emptyset.$$ For each $n\in\omega$, since $K_{n+1}^{\ast}\cdot \overline{V_{n}}$ is compact, it follows from Lemma~\ref{compact-separated} that there is an open neighborhood $U$ of $e$ in $G$ such that, for each $i=0, 1, \cdots, n+1$, we have$$(K_{i}^{\ast}\cdot ((\overline{V_{n}}\cdot U)\cdot U))\cap C_{i+1}=\emptyset.$$Put $$V_{n+1}=(V_{n}\cdot U)\cap K_{n+1}^{\ast}.$$ By Lemma~\ref{dense}, we have $$\overline{V_{n+1}}\subset V_{n+1}\cdot U \subset (V_{n}\cdot U)\cdot U,$$ whence it follows that the set $V_{n+1}$ satisfies (3). Obviously, $V_{n+1}$ satisfies (1) and (2).
\smallskip

Put $V=\bigcup_{n=0}^{\infty}V_{n}.$ By (1) and (2), we have $$V\cap K_{n}=\bigcup_{i=0}^{\infty}(V_{i}\cap K_{n})=\bigcup_{i=n}^{\infty}(V_{i}\cap K_{n})$$ is open in $K_{n}$ for each $n\in\omega$. Since the family $\{K_{n}: n\in\omega\}$ determines the topology of $G$, we see that $V$ is open in $G$. From (3) and the definition of $V$ it follows that $$(K_{j}\cdot V)\cap C_{j+1}=\emptyset, j\in\omega.$$ Since $\zeta$ is a left Cauchy filter in $G$, we can take a $C\in\zeta$ and $x\in G$ such that $q(x, C)\subset V$, which implies that $C\subset x\cdot V$. Hence there exists an integer $n\geq 1$ such that $x\in K_{n}$, then $x\cdot V\cap C_{n+1}=\emptyset$. Since $C\subset x\cdot V$, and therefore $C\cap C_{n+1}=\emptyset$, which is a contradiction with $C, C_{n+1}\in\zeta$.

Therefore,  $G$ is a rectifiable complete space.
\end{proof}


\begin{thebibliography}{99}
\bibitem{A2008} A. Arhangel'skii and M. Tkachenko,
  Topological groups and related structures, Atlantis Press, Paris; World
  Scientific Publishing Co. Pte. Ltd., Hackensack, NJ, 2008.

\bibitem{A20092} A.V. Arhangel'skii, {\it Topological spaces with flexible diagonal}, {\it Questions Answers Gen. Topology},
{\bf 27} (2009), 83-105.

\bibitem{A2009} A.V. Arhangel'skii and M.M. Choban,  {\it Remainders of rectifiable spaces}, {\it Topology Appl.},
{\bf 157} (2010), 789-799.

\bibitem{A2014} A.V. Arhangel'skii and J.van Mill,  {\it Topological groups with a $bc$-base}, {\it Topology Appl.}, {\bf 179} (2015), 5-12.

\bibitem{C1992} M.M. $\check{C}$hoban, {\it The structure of locally compact algebras}, {\it Serdica}, {\bf 18} (1992), 129-137.

\bibitem{E1989} R. Engelking, {\it General topology} (revised and completed edition), Heldermann
Verlag, Berlin, 1989.

\bibitem{G1996} A.S. Gul$^{\prime}$ko, {\it Rectifiable spaces}, {\it Topology Appl.}, {\bf 68} (1996), 107-112.

\bibitem{HK} K.H. Hofmann and J.R. Martin, {\it Topological Left-loops}, {\it Topology Proc.}, {\bf 39} (2012), 185--194.

\bibitem{LF4} F. Lin, {\it The operators of rotoids in homogeneous spaces or generalized ordered spaces}, {\it Topology Appl.}, {\bf 164} (2014), 248--262.

\bibitem{LF3} F. Lin, {\it Metrizability of rectifiable spaces}, {\it Bull. Malays. Math. Sci. Soc.}, {\bf 36} (2013), 1099--1107.

\bibitem{LF2} F. Lin, {\it Topologically subordered rectifiable spaces and compactifications}, {\it Topology Appl.}, {\bf 159} (2012), 360--370.

\bibitem{S1} F. Lin, S. Lin and I. S\'{a}nchez, {\it A note on pseudobounded paratopological groups}, {\it Topology Algebra Appl.}, {\bf } (2014), 11--18.

\bibitem{LF1} F. Lin and R. Shen,  {\it On rectifiable spaces and paratopological groups}, {\it Topology Appl.},
{\bf 158} (2011), 597-610.

\bibitem{LLL} F. Lin, C. Liu and S. Lin,  {\it A note on rectifiable spaces}, {\it Topology Appl.}, {\bf 159} (2012), 2090--2101.

\bibitem{PL}  L.X. Peng and S.J. Guo,  {\it Two questions on rectifiable spaces and related conclusions}, {\it Topology Appl.}, {\bf 159} (2012), 3335-3339.

\bibitem{V1989} V.V. Uspenskij, {\it The Mal$^{\prime}$tsev operation on countably compact spaces},
 {\it Comment. Math. Univ. Carolin.}, {\bf 30}(1989), 395-402.

\bibitem{V1990} V.V. Uspenskij, {\it Topological groups and Dugundji compacta}, {\it Mat. Sb.}
{\bf 180} (1989), 1092--1118 (Russian); English transl. in:
{\it Math. USSR-Sb.} {\bf 67} (1990), 555--580.
\end{thebibliography}
\end{document}